\documentclass[11pt]{article}

\usepackage{amssymb, amsmath, amsthm, graphicx}
\usepackage[left=1in,top=1in,right=1in]{geometry}

\usepackage{subfigure}

\date{}

\theoremstyle{plain}
      \newtheorem{theorem}{Theorem}[section]
      \newtheorem{lemma}[theorem]{Lemma}

      \newtheorem{observation}[theorem]{Observation}

\theoremstyle{definition}

\theoremstyle{remark}

\def\conv{\mbox{\rm conv}}

\title{On the Erd\H os-Szekeres convex polygon problem}
\author{Andrew Suk\thanks{University of Illinois, Chicago, IL.  Supported by NSF grant DMS-1500153.
  Email: {\tt suk@uic.edu}.}  }

\begin{document}

\maketitle

\begin{abstract}
Let $ES(n)$ be the smallest integer such that any set of $ES(n)$ points in the plane in general position contains $n$ points in convex position.   In their seminal 1935 paper, Erd\H os and Szekeres showed that $ES(n) \leq {2n  - 4\choose n-2} + 1 = 4^{n -o(n)}$.  In 1960, they showed that $ES(n) \geq 2^{n-2} + 1$ and conjectured this to be optimal. In this paper, we nearly settle the Erd\H os-Szekeres conjecture by showing that $ES(n) =2^{n +o(n)}$.

\end{abstract}

\section{Introduction}

In their classic 1935 paper, Erd\H os and Szekeres \cite{es} proved that, for every integer $n\geq 3$, there is a minimal integer $ES(n)$, such that any set of $ES(n)$ points in the plane in general position\footnote{No three of the points are on a line.} contains $n$ points in convex position, that is, they are the vertices of a convex $n$-gon.

Erd\H os and Szekeres gave two proofs of the existence of $ES(n)$.  Their first proof used a quantitative version of Ramsey's Theorem, which gave a very poor upper bound for $ES(n)$.  The second proof was more geometric and showed that $ES(n) \leq {2n-4\choose n-2} + 1$ (see Theorem \ref{cupscaps} in the next section).  On the other hand, they showed that $ES(n) \geq 2^{n-2} + 1$ and conjectured this to be sharp~\cite{es2}.

Small improvements have been made on the upper bound ${2n - 4\choose n-2} + 1 \approx \frac{4^n}{\sqrt{n}}$ by various researchers \cite{CG,KP,tv1,tv2,MV,ny}, but no improvement in the order of magnitude has ever been made.  The most recent upper bound, due to Norin and Yuditsky \cite{ny} and Mojarrad and Vlachos \cite{MV}, says that

$$\lim\sup\limits_{n\rightarrow \infty} \frac{ES(n)}{{2n- 4\choose n-2}} \leq \frac{7}{16}.$$

\noindent  In the present paper, we prove the following.

 \begin{theorem}\label{main} For all $n \geq n_0$, where $n_0$ is a large absolute constant, $ES(n) \leq 2^{n + 6n^{2/3}\log n}.$

 \end{theorem}

The study of $ES(n)$ and its variants\footnote{Higher dimensions \cite{K,KV,suk}, for families of convex bodies in the plane \cite{fpss,dhh}, etc.} has generated a lot of research over the past several decades.  For a more thorough history on the subject, we refer the interested reader to \cite{MS,BPM,tv1}.  All logarithms are to base 2.

\section{Notation and tools}

In this section, we recall several results that will be used in the proof of Theorem \ref{main}.  We start with the following simple lemma.

\begin{lemma}[See Theorem 1.2.3 in \cite{matousek}]\label{four}
Let $X$ be a finite point set in the plane in general position such that every four members in $X$ are in convex position.  Then $X$ is in convex position.

\end{lemma}

The next theorem is a well-known result from \cite{es}, which is often referred to as the Erd\H os-Szekeres cups-caps theorem. Let $X$ be a $k$-element point set in the plane in general position. We say that $X$ forms a \emph{$k$-cup} (\emph{$k$-cap}) if $X$ is in convex position and its convex hull is bounded above (below) by a single edge.  In other words, $X$ is a cup (cap) if and only if for every point $p \in X$, there is a line $L$ passing through it such that all of the other points in $X$ lie on or above (below) $L$.  See Figure \ref{cupcappic}.

\begin{figure}
\begin{center}
\includegraphics[width=280pt]{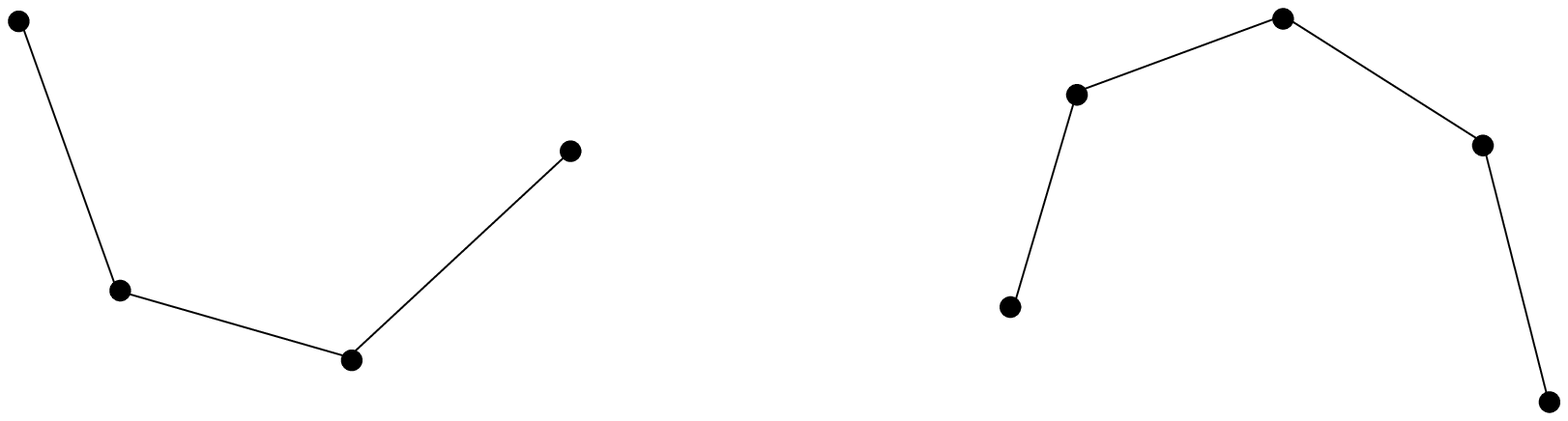}
  \caption{A 4-cup and a 5-cap.}\label{cupcappic}
 \end{center}
\end{figure}

\begin{theorem}[\cite{es}]\label{cupscaps}
Let $f(k,\ell)$ be the smallest integer $N$ such that any $N$-element planar point set in the plane in general position contains a $k$-cup or an $\ell$-cap.  Then

$$f(k,\ell) = {k + \ell - 4\choose k - 2} + 1.$$

\end{theorem}

The next theorem is a combinatorial reformulation of Theorem \ref{cupscaps} observed by Hubard et al.~\cite{hubard} (see also \cite{fpss,ms}).  A transitive 2-coloring of the triples of $\{1,2,\ldots,N\}$ is a 2-coloring, say with colors red
and blue, such that, for $i_1 < i_2 < i_3 < i_4$, if triples $(i_1, i_2, i_3)$ and $(i_2, i_3, i_4)$ are red (blue), then $(i_1, i_2, i_4)$ and $(i_1, i_3, i_4)$ are also red (blue).

\begin{theorem}[\cite{es}]\label{comcupcap}

Let $g(k,\ell)$ denote the minimum integer $N$ such that, for every transitive 2-coloring on the triples of $\{1,2,\ldots,N\}$, there exists a red clique of size $k$ or a blue clique of size $\ell$.  Then

$$g(k,\ell) = f(k,\ell)  = {k + \ell - 4\choose k - 2} + 1.$$
\end{theorem}

The next theorem is due to P\'or and Valtr \cite{PV}, and is often referred to as the positive-fraction Erd\H os-Szekeres theorem (see also \cite{BV,PS}).   Given a $k$-cap ($k$-cup) $X = \{x_1,\ldots, x_{k}\}$, where the points appear in order from left to right, we define the \emph{support of} $X$ to be the collection of open regions $\mathcal{C} = \{T_1,\ldots, T_{k}\}$, where $T_i$ is the region outside of $\conv(X)$ bounded by the segment $\overline{x_ix_{i + 1}}$ and by the lines $x_{i-1}x_{i}$, $x_{i+1}x_{i+2}$ (where $x_{k +1} = x_1$, $x_{k + 2} = x_2$, etc.).  See Figure \ref{tbase}.

\begin{figure}
\begin{center}
\includegraphics[width=240pt]{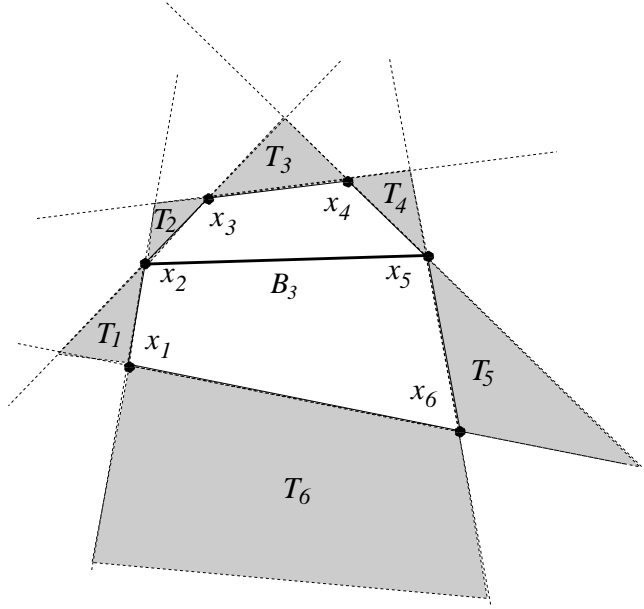}
  \caption{Regions $T_1,\ldots, T_6$ in the support of $X = \{x_1,\ldots, x_6\}$, and segment $B_3$.}\label{tbase}
 \end{center}
\end{figure}

\begin{theorem}[Proof of Theorem 4 in \cite{PV}]\label{partition}

Let $k\geq 3$ and let $P$ be a finite point set in the plane in general position such that $|P| \geq 2^{32k}$.  Then there is a $k$-element subset $X\subset P$ such that $X$ is either a $k$-cup or a $k$-cap, and the regions $T_1,\ldots, T_{k-1}$ from the support of $X$ satisfy $|T_i\cap P| \geq \frac{|P|}{2^{32k}}.$  In particular, every $(k-1)$-tuple obtained by selecting one point from each $T_i\cap P$, $i = 1,\ldots, k-1$, is in convex position.

\end{theorem}

Note that Theorem \ref{partition} does not say anything about the points inside region $T_{k}$.  Let us also remark that in the proof of Theorem \ref{partition} in \cite{PV}, the authors find a $2k$-element set $X\subset P$, such that $k$ of the regions in the support of $X$ each contain at least $\frac{|P|}{2^{32k}}$ points from $P$, and therefore these regions may not be consecutive.  However, by appropriately selecting a $k$-element subset $X'\subset X$, we obtain Theorem \ref{partition}.

\section{Proof of Theorem \ref{main}}
Let $P$ be an $N$-element planar point set in the plane in general position, where $N =  \lfloor 2^{n+ 6 n^{2/3}\log n}\rfloor $ and $n \geq n_0$, where $n_0$ is a sufficiently large absolute constant.  Set $k = \lceil n^{2/3}\rceil$.   We apply Theorem~\ref{partition} to $P$ with parameter $k  + 3 $, and obtain a subset $X = \{x_1,\ldots, x_{k+3}\}\subset P$ such that $X$ is a cup or a cap, and the points in $X$ appear in order from left to right.  Moreover since $k = \lceil n^{2/3}\rceil$ is large, regions $T_1,\ldots, T_{k+2}$ in the support of $X$ satisfy

$$|T_i\cap P| \geq \frac{N}{2^{40k}}.$$   Set $P_i = T_i\cap P$ for $i = 1,\ldots, k+2$.  We will assume that $X$ is a cap, since a symmetric argument would apply.  We say that the two regions $T_i$ and $T_j$ are \emph{adjacent} if $i$ and $j$ are consecutive indices.

Consider the subset $P_i\subset P$ and the region $T_i$, for some fixed $i \in \{2,\ldots, k+1\}$.  Let $B_i$ be the segment $\overline{x_{i-1}x_{i+2}}$.   See Figure~\ref{tbase}. The point set $P_i$ naturally comes with a partial order $\prec$, where $p\prec q$ if $p\neq q$ and $q \in \conv(B_i\cup p)$.  Set $\alpha = 3n^{-1/3}\log n$.  By Dilworth's Theorem \cite{dil}, $P_i$ contains either a chain of size at least $|P_i|^{1-\alpha}$ or an antichain of size at least $|P_i|^{\alpha}$ with respect to $\prec$.  The proof now falls into two cases.

\medskip

\noindent \emph{Case 1}.  Suppose there are $t = \lceil\frac{n^{1/3}}{2}\rceil$ parts $P_i$ in the collection $\mathcal{F}= \{P_2,P_3,\ldots, P_{k+1}\}$, such that no two of them are in adjacent regions, and each such part contains a subset $Q_i$ of size at least $|P_i|^{\alpha}$ such that $Q_i$ is an antichain with respect to $\prec$.  Let $Q_{j_1},Q_{j_2},\ldots, Q_{j_t}$ be the selected subsets.

For each $Q_{j_r}$, $r\in \{1,\ldots, t\}$, the line spanned by any two points in $Q_{j_r}$ does not intersect the segment $B_{j_r}$, and therefore, does not intersect region $T_{j_{w}}$ for $w\neq r$ (by the non-adjacency property).  Since $n$ is sufficiently large, we have $40k < n^{2/3}\log n$, and therefore

$$|Q_{j_r}|  \geq |P_i|^{\alpha} \geq \left(\frac{N}{2^{40k}}\right)^{\alpha} \geq 2^{3n^{2/3}\log n + 15n^{1/3}\log^2n} \geq {n + \lceil 2n^{2/3}\rceil - 4\choose n - 2} +1= f(n,\lceil 2n^{2/3}\rceil).$$

\noindent Theorem \ref{cupscaps} implies that $Q_{j_r}$ contains either an $n$-cup or a $\lceil 2n^{2/3}\rceil$-cap.  If we are in the former case for any $r \in \{1,\ldots , t\}$, then we are done.  Therefore we can assume $Q_{j_r}$ contains a subset $S_{j_r}$ that is a $\lceil 2n^{2/3}\rceil$-cap, for all $r \in \{1,\ldots, t\}$.

We claim that $S = S_{j_1}\cup \cdots \cup S_{j_{t}}$ is a cap, and therefore $S$ is in convex position.  Let $p \in S_{j_r}$.  Since $|S_{j_r}| \geq 2$, there is a point $q \in S_{j_r}$ such that the line $L$ supported by the segment $\overline{pq}$ has the property that all of the other points in $S_{j_r}$ lie below $L$.  Since $L$ does not intersect $B_{j_r}$, all of the points in $S\setminus \{p,q\}$ must lie below $L$.  Hence, $S$ is a cap and

$$|S| = |S_{j_1}\cup \cdots \cup S_{j_{t}}| \geq \frac{n^{1/3}}{2}(2n^{2/3}) = n.$$

\medskip

\noindent \emph{Case 2}.  Suppose we are not in Case 1.  Then there are $\lceil n^{1/3}\rceil$ consecutive indices $j, j+1, j+2, \ldots,$ such that each such part $P_{j+r}$ contains a subset $Q_{j + r}$ such that $Q_{j + r}$ is a chain of length at least $|P_{j + r}|^{1-\alpha}$ with respect to $\prec$.  For simplicity, we can relabel these sets $Q_1,Q_2,Q_3,\ldots$.

Consider the subset $Q_i$ inside the region $T_i$, and order the elements in $Q_i = \{p_1,p_2,p_3,\ldots\}$ with respect to $\prec$.  We say that $Y \subset Q_i$ is a \emph{right-cap} if $x_{i}\cup Y$ is in convex position, and we say that $Y$ is a \emph{left-cap} if $x_{i +1}\cup Y$ is in convex position.  Notice that left-caps and right-caps correspond to the standard notion of cups and caps after applying an appropriate rotation to the plane so that the segment $\overline{x_ix_{i + 1}}$ is vertical.  Since $Q_i$ is a chain with respect to $\prec$, every triple in $Q_i$ is either a left-cap or a right-cap, but not both.  Moreover, for $i_1 < i_2 < i_3 < i_4$, if $(p_{i_1},p_{i_2},p_{i_3})$ and $(p_{i_2},p_{i_3},p_{i_4})$ are right-caps (left-caps), then $(p_{i_1},p_{i_2},p_{i_4})$ and $(p_{i_1},p_{i_3},p_{i_4})$ are both right-caps (left-caps).  By Theorem~\ref{comcupcap}, if $|Q_i| \geq f(k,\ell)$, then $Q_i$ contains either a $k$-left-cap or an $\ell$-right-cap.  We make the following observation.

\begin{observation}\label{glue}

Consider the (adjacent) sets $Q_{i-1}$ and $Q_{i}$.  If $Q_{i-1}$ contains a $k$-left-cap $Y_{i-1}$, and $Q_{i}$ contains an $\ell$-right-cap $Y_{i}$, then $Y_{i-1}\cup Y_{i}$ forms $k+\ell$ points in convex position.

\end{observation}

\begin{proof}
By Lemma \ref{four}, it suffices to show every four points in $Y_{i-1}\cup Y_{i}$ are in convex position.  If all four points lie in $Y_i$, then they are in convex position.  Likewise if they all lie in $Y_{i - 1}$, they are in convex position.  Suppose we take two points $p_1,p_2 \in Y_{i-1}$ and two points $p_3,p_4 \in Y_{i}$.  Since $Q_{i-1}$ and $Q_i$ are both chains with respect to $\prec$, the line spanned by $p_1,p_2$ does not intersect the region $T_{i}$, and the line spanned by $p_3,p_4$ does not intersect the region $T_{i-1}$.  Hence $p_1,p_2,p_3,p_4$ are in convex position.  Now suppose we have $p_1,p_2,p_3 \in Y_{i-1}$ and $p_4 \in Y_{i}$.  Since the three lines $L_1,L_2,L_3$ spanned by $p_1,p_2,p_3$ all intersect the segment $B_{i-1}$, both $x_{i}$ and $p_4$ lie in the same region in the arrangement of $L_1\cup L_2\cup L_3$.  Therefore $p_1,p_2,p_3,p_4$ are in convex position. The same argument follows in the case that $p_1 \in Y_{i-1}$ and $p_2,p_3,p_4 \in Y_i$.  See Figure \ref{gluepic}.  \end{proof}

\begin{figure}\label{gluepic}
\begin{center}
\includegraphics[width=200pt]{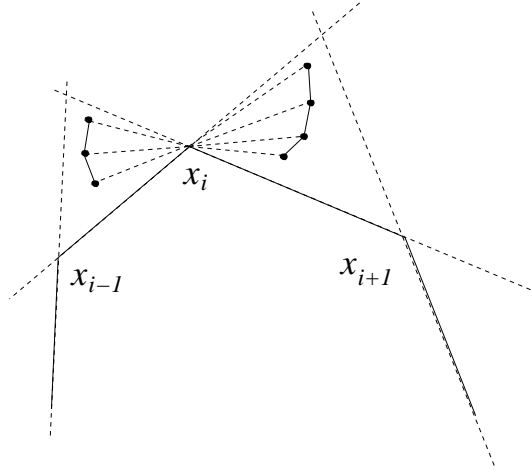}
  \caption{A 3-left-cap in $Q_{i-1}$ and a 4-right-cap in $Q_{i}$, which forms 7 points in convex position.}\label{gluepic}
 \end{center}
\end{figure}

\noindent We have for $i \in \{1,\ldots, \lceil n^{1/3}\rceil\}$,

\begin{equation}\label{p1}|Q_i| \geq |P_i|^{(1 - \alpha)} \geq \left(\frac{N}{2^{40k}}\right)^{1 - \alpha} \geq  2^{n + 2n^{2/3}\log n - 15n^{1/3}\log^2n}.\end{equation}

Set $K = \lceil n^{2/3}\rceil$.  Since $n$ is sufficiently large, we have $$|Q_1| \geq {n + K -4\choose K - 2} +1= f(K,n),$$ which implies that $Q_1$ either contains an $n$-right-cap, or a $K$-left-cap.  In the former case we are done, so we can assume that $Q_1$ contains a $K$-left-cap.  Likewise, $|Q_2| \geq {n + K - 4\choose 2K - 2} +1= f(2K,n-K)$, which implies $Q_2$ contains either an $(n-K)$-right-cap, or a $(2K)$-left-cap.  In the former case we are done since Observation \ref{glue} implies that the $K$-left-cap in $Q_1$ and the $(n-K)$-right-cap in $Q_2$ form $n$ points in convex position.  Therefore we can assume $Q_2$ contains a $(2K)$-left-cap.

In general, if we know that $Q_{i-1}$ contains an $(iK-K)$-left-cap, then we can conclude that $Q_i$ contains an $(iK)$-left-cap.  Indeed, for all $i \leq  \lceil n^{1/3}\rceil$ we have

\begin{equation}\label{p2} {n + K - 4\choose iK - 2} \leq   2^{n + \lceil n^{2/3}\rceil -4}.   \end{equation}

\noindent Since $n$ is sufficiently large, (\ref{p1}) and (\ref{p2}) imply that

$$|Q_i|  \geq  2^{n + 2n^{2/3}\log n - 15n^{1/3}\log^2n} \geq {n + K - 4 \choose iK -2} +1= f(iK, n-iK + K ).$$

\noindent Therefore, $Q_i$ contains either an $(n-iK + K)$-right-cap or an $(iK)$-left-cap.  In the former case we are done by Observation \ref{glue} (recall that we assumed $Q_{i-1}$ contains an $(iK-K)$-left-cap), and therefore we can assume $Q_i$ contains an $(iK)$-left-cap.  Hence for $i = \lceil n^{1/3} \rceil$, we can conclude that $Q_{\lceil n^{1/3} \rceil}$ contains an $n$-left-cap.  This completes the proof of Theorem~\ref{main}.\qed

\section{Concluding remarks}

Following the initial publication of this work on arXiv, we have learned that G\'abor Tardos has improved the lower order term in the exponent, showing that $ES(n) = 2^{n  + O(\sqrt{n\log n})}$.

\medskip

\medskip

\medskip

\noindent \textbf{Acknowledgments.} I would like to thank J\'anos Pach,  G\'abor Tardos,  and G\'eza T\'oth for numerous helpful comments.


\begin{thebibliography}{99}

\bibitem{BV} I. B\'ar\'any and P. Valtr, A positive fraction Erd\H os-Szekeres theorem, \emph{Discrete Comput. Geom.} \textbf{19} (1998), 335--342.

\bibitem{BPM} P. Brass, W. Moser, and J. Pach, Convex polygons and the Erd\H os Szekeres problem, Chapter 8.2 in \emph{Research problems in discrete geometry}, Springer, (2005).

\bibitem{CG} F. R. K. Chung and R. L. Graham, Forced convex $n$-gons in the plane, \emph{Discrete Comput. Geom.} \textbf{19} (1998), 367--371.

\bibitem{dil} R. Dilworth, A decomposition theorem for partially ordered sets, \emph{Ann. of Math.} \textbf{51} (1950), 161--166.


\bibitem{dhh} M. Dobbins, A. Holmsen, and A. Hubard, The Erd\H os-Szekeres problem for non-crossing convex sets, \emph{Mathematika} \textbf{60} (2014), 463--484.



\bibitem{e} P. Erd\H os, Some of my favorite problems and results, \emph{The Mathematics of Paul Erd\H os, Vol. I} (R. L. Graham and J. Nesetril, eds.), Springer-Verlag, Berlin, 1997, 47--67.


\bibitem{es} P. Erd\H os and G. Szekeres, A combinatorial problem in geometry, \emph{Compos. Math.} \textbf{2} (1935), 463--470.

\bibitem{es2} P. Erd\H os and G. Szekeres, On some extremum problems in elementary geometry, \emph{Ann. Univ. Sci. Budapest.  E\"otv\"os Sect. Math.} \textbf{3-4} (1960/1961), 53--62.

\bibitem{fpss} J. Fox, J. Pach, B. Sudakov, and A. Suk, Erd\H os-Szekeres-type theorems for monotone paths and convex bodies, \emph{Proc. Lond. Math. Soc.} \textbf{105} (2012), 953--982.


\bibitem{hubard} A. Hubard, L. Montejano, E. Mora, and A. Suk, Order types of convex bodies, \emph{Order} \textbf{28} (2011), 121--130.

\bibitem{K} G. K\'arolyi, Ramsey-remainder for convex sets and the Erd\H os-Szekeres theorem, \emph{Discrete Appl. Math.} \textbf{109} (2001), 163--175.

\bibitem{KV} G. K\'arolyi and P. Valtr, Point configurations in $d$-space without large subsets in convex position, \emph{Discrete Comput. Geom.} \textbf{30} (2003) 277--286.


\bibitem{KP} D. Kleitman and L. Pachter, Finding convex sets among points in the plane, \emph{Discrete Comput. Geom.} \textbf{19} (1998), 405--410.


\bibitem{matousek} J.~Matou\v{s}ek, \emph{Lectures on Discrete Geometry}, Springer-Verlag New York, Inc., 2002.

\bibitem{MS} W. Morris and V. Soltan, The Erd\H os-Szekeres problem on points in convex
position - a survey, \emph{Bull. Amer. Math. Soc.} \textbf{37} (2000), 437--458.

\bibitem{ms} G. Moshkovitz and A. Shapira, Ramsey Theory, integer partitions and a new proof of the Erd\H os-Szekeres Theorem, \emph{Adv. Math.} \textbf{262} (2014), 1107--1129.

\bibitem{MV} H. Mojarrad and G. Vlachos, On the Erd\H os-Szekeres conjecture, \emph{Discrete Comput. Geom.} \textbf{56} (2016), 165--180.

\bibitem{ny} S. Norin and Y. Yuditsky, Erd\H os-Szekeres without induction, http://arxiv.org/abs/1509.03332, (2015).

\bibitem{PS} J. Pach and J. Solymosi, Canonical theorems for convex sets, \emph{Discrete Comput. Geom.} \textbf{19} (1998), 427--435.

 \bibitem{PV} A. P\'or and P. Valtr, The partitioned version of the Erd\H os-Szekeres theorem, \emph{Discrete Comput. Geom.} \textbf{28} (2002), 625--637.

 \bibitem{suk} A. Suk, A note on order-type homogeneous point sets, \emph{Mathematika} \textbf{60} (2014), 37--42.

\bibitem{tv1} G. T\'oth and P. Valtr, The Erd\H os-Szekeres theorem: Upper bounds and related results, \emph{Combinatorial and Computational Geometry} (J.E. Goodman et al., eds.), Publ. M.S.R.I. \textbf{52} (2006) 557--568.

\bibitem{tv2} G. T\'oth and P. Valtr, Note on the Erd\H os-Szekeres theorem, \emph{Discrete Comput. Geom.} \textbf{19} (1998), 457--459.


 \end{thebibliography}
\end{document}